\documentclass[11pt]{amsart}
\usepackage{graphicx}
\newtheorem{thm}{Theorem}[section]
\newtheorem{cor}[thm]{Corollary}

\newtheorem{prop}[thm]{Proposition}
\theoremstyle{definition}

\theoremstyle{remark}

\numberwithin{equation}{section}

\begin{document}

\title[ Wilson's functional equation]{Solutions and  stability of a generalization of Wilson's equation}%
\author[B. Bouikhalene and E. Elqorachi]{Bouikhalene Belaid and Elqorachi Elhoucien}%
\thanks{\textit{Key words and phrases}, D'Alembert's functional equation, Locally compact group, Involution, Character, Complex measure,  Wilson's functional equation, Hyers-Ulam stability. }%


\begin{abstract}
In this paper we study the solutions and stability of the
generalized Wilson's functional equation
$\int_{G}f(xty)d\mu(t)+\int_{G}f(xt\sigma(y))d\mu(t)=2f(x)g(y),\;
x,y\in G$,  where $G$ is a locally compact group, $\sigma$ is a
continuous  involution of $G$  and $\mu$ is an idempotent complex
measure with compact support and which is  $\sigma$-invariant. We
show that
$\int_{G}g(xty)d\mu(t)+\int_{G}g(xt\sigma(y))d\mu(t)=2g(x)g(y),\;
x,y\in G$ if $f\neq 0$ and $\int_{G}f(t.)d\mu(t)\neq 0$. We also
study some stability theorems of that equation and we establish  the
stability on noncommutaive groups of the classical Wilson's
functional equation $f(xy)+\chi(y)f(x\sigma(y))=2f(x)g(y)\; x,y\in
G$ , where $\chi$ is a unitary character of $G$.
\end{abstract}
\maketitle
\section{Introduction and preliminaries }Althought d'Alembert's
functional equation \begin{equation}\label{eq221} f (x +y)+f (x -y)
= 2f (x)f (y)\;\text{for all }\;x,y \in \mathbb{R}\end{equation} for
functions $f$: $\mathbb{R}\longrightarrow \mathbb{C}$ on the real
line has it roots back in d'Alembert's investigation of vibrating
strings \cite{d1} from 1975. Furthermore, one solution of
(\ref{eq221}) is $f(x)=\cos(x)$, another $f(x)=\cosh(x)$. The
obvious extension of (\ref{eq221}) from $\mathbb{R}$ to an abelian
group ($G,+$) is the functional equation
\begin{equation}\label{eq222} f (x +y)+f (x -y)
= 2f (x)f (y)\;\text{for all }\;x,y \in {G},\end{equation} where
$f$: $G\longrightarrow \mathbb{C}$ is the unknown. The non-zero
solution of equation (\ref{eq222}) are of the form
$f(x)=\frac{\chi(x)+\chi(-x)}{2}$, $x\in G,$ where $\chi$ is a
character in $G$. The result is obtained by Kannappan in \cite{k1}.
If the group $G$ is not assumed abelian the solutions of equation
\begin{equation}\label{eq223} f (xy)+f (xy^{-1})
= 2f (x)f (y)\;\text{for all }\;x,y \in {G}\end{equation} are
obtained by Davison \cite{davis1,davis2}. There are of the form
$f=\frac{1}{2}tr(\varrho)$, where $\varrho$ is continuous
algebraicaly
irreducible representation of $G$ on $\mathbb{C}^{2}$.\\
In \cite{w} Wilson dealt with functional equations related to and
generalizing (\ref{eq221}) on the real line. He generalized the
d'Alembert's functional equation (\ref{eq221}) to
\begin{equation}\label{eq224} f (x+y)+f (x-y)
= 2f (x)g (y)\;\text{for all }\;x,y \in \mathbb{R}.\end{equation}
Let us not that if $f\neq 0$ is a solution of equation
(\ref{eq224}), then $g$ satisfies equation (\ref{eq221}). \\Some
general properties of the solutions of equation
\begin{equation}\label{eq225} f (xy)+f (x\sigma(y))
= 2f (x)g (y)\;\text{for all }\;x,y \in {G}\end{equation}on  a
topological monoid equipped with a continuous involution $\sigma$
can be found in \cite{st1}.\\Recently, Ebanks, Stetk\ae r \cite{st2}
and Stetk\ae r \cite{st3} proved a natural interesting relation
between  Wilson's functional equation  (\ref{eq225})
 and d'Alembert's functional equation \begin{equation}\label{eq226} f (xy)+f (x\sigma(y))
= 2f (x)f (y)\;\text{for all }\;x,y \in {G}\end{equation} and for
$\sigma(x)=x^{-1}$. That is if $f\neq 0$ is a solution of equation
(\ref{eq225}), the $g$ is a solution of equation (\ref{eq226}).\\
The Hyers-Ulam stability of d'Alembert's functional equation
(\ref{eq222})
 was investigated by J.A. Baker in
\cite{ba1}. In \cite{ba2}, J. Baker, J. Lawrence and F. Zorzitto
introduced the superstability of the exponential equation $f (x + y)
= f (x)f (y),x,y\in G$. Badora \cite{bad4} gave a new, shorter proof
of Baker's result.  A different generalization of the result of
Baker, Lawrence and Zorzitto was given by L. Sz\'ekelyhidi  \cite{s1}.\\
On abelian groups, the stability of d'Alembert's functional equation
(\ref{eq222}) and
 Wilson's functional
equation (\ref{eq224}) and other functional equation has been
investigated by several authors. The interested reader should refer
to \cite{badora2}, \cite{baa3}, \cite{b}, \cite{sahoo}, \cite{elq4},
\cite{g1}, \cite{h1}, \cite{kim3},  \cite{kim1}, \cite{kim2},
\cite{mos10}, \cite{16}, \cite{rd1}, \cite{s2} and \cite{24}, for a
thorough account on the subject of stability of functional equations.\\
The aim of this paper is to study some properties of the solutions
and Hyers-Ulam stability of some generalization of d'Alembert's and
Wilson's functional equations which has been introduced in
\cite{elq1}. As an application we obtain the Hyers-Ulam stability of
Wilson's functional  equation (\ref{eq225}) on groups that need not
be abelian.\\
\\Throughout this paper, we Let $G$ be a locally compact group, $C(G)$ the complex algebra of
all continuous complex valued functions on $G$. $M(G)$ the Banach
algebra of the complex bounded measures on $G$. It's the
topolological dual of $C_{0}(G):$ The Banach space of continuous
functions vanishing at infinity. Let $\sigma$: $G\longrightarrow G$
be a continuous involution of $G$, that is
$\sigma(xy)=\sigma(y)\sigma(x)$ and $\sigma(\sigma(x)) =x$ for all
$x,y\in G$. If $\mu\in M(G)$ is a measure with copmact support, we
let $\mu_{\sigma}$ denote the complex measure with compact support
and defined by the relation: $<\mu_{\sigma},f>=<\mu,f\circ \sigma>$,
$f\in C(G)$, where $<\mu,f>=\int_{G}f(t)d\mu(t)$. We will say that
$\mu$ is $\sigma$-invariant if $\mu=\mu_{\sigma}$. We recall that
the convolution measure $\mu\ast\mu$ is the measure defined on
$C(G)$ by $<\mu\ast\mu,f>=\int_{G}\int_{G}f(ts)d\mu(t)d\mu(s)$.
Finally, for a
 continuous function $f$ we let $f_\mu(x)=\int_{G}f(tx)d\mu(t),$ $x\in G$ and we say that $f$ is $\mu$-biinvariant, if
$\int_{G}f(xt)d\mu(t)=\int_{G}f(tx)d\mu(t)=f(x)$, for all $x\in
G$.\section{Relations between Wilson's and d'Alembert's functional
equations} In the special case where  $\mu$ is a Gelfand measure or
$f$ satisfies some Kannappan type condition Elqorachi and Akkouchi
[\cite{elq1}, Proposition 3.2] obtained a natural relation between
the generalized Wilson's functional equation
\begin{equation}\label{eq11}
\int_{G}f(xty)d\mu(t)+\int_{G}f(xt\sigma(y))d\mu(t)=2f(x)g(y),\;
x,y\in G\end{equation}
 and the generalized d'Alembert's short functional
equation \begin{equation}\label{eq12}
\int_{G}g(xty)d\mu(t)+\int_{G}g(xt\sigma(y))d\mu(t)=2g(x)g(y),\;
x,y\in G\end{equation} That is if the pair  $f,g$: $G\longrightarrow
\mathbb{C}$, where $f\neq 0$, is a solution of generalized Wilson's
functional equation (\ref{eq11}) then $g$ is a solution of the
generalized d'Alembert's functional equation (\ref{eq12}). In more
general setting the authors [\cite{elq2}, Corollary 2.7 (iii)] got a
weaker result, that is if the pair $f,g$: $G\longrightarrow
\mathbb{C}$, where $f\neq 0$, is a solution of generalized Wilson's
functional equation (\ref{eq11}) then $g$ is a solution of the
generalized d'Alembert's long functional equation
\begin{equation}\label{eq13}
\int_{G}g(xty)d\mu(t)+\int_{G}g(ytx)d\mu(t)+\int_{G}g(xt\sigma(y))d\mu(t)+\int_{G}g(\sigma(y)tx)d\mu(t)\end{equation}$$=4g(x)g(y),\;
x,y\in G.$$ The following theorem is a generalization of the result
obtained by Ebanks, Stetk\ae r  \cite{st2}.
\begin{thm}Let $\sigma$ be a continuous involution of $G$. Let $\mu$
be a complex measure with compact support and which  is
$\sigma$-invariant. If the pair $f,g$: $G\longrightarrow
\mathbb{C}$, where $f\neq 0$ is a  continuous solution of the
generalized Wilson's functional equation (\ref{eq11}) and $f$ is
odd: $f(\sigma(x))=-f(x)$ for all $x\in G.$  Then $g$ is a solution
of the generalized d'Alembert's short functional equation
(\ref{eq12}).
\end{thm}\begin{proof} The  proof  is closely related to the one obtained by  Ebanks, Stetk\ae r
\cite{st2}.
 Let us assume that the pair $f,g$ is a solution of equation
(\ref{eq11}) with $f\neq 0$ and  $f(\sigma(x))=-f(x)$ for all $x\in
G$. By replacing $y$ by $\sigma(y)$ in  (\ref{eq11}) we get
$g(\sigma(x))=g(x)$ for all $x\in G.$  Now, we consider the new
function
$$\Psi(x,y)=\int_{G}g(xty)d\mu(t)+\int_{G}g(\sigma(y)tx)d\mu(t)-2g(x)g(y),\;x,y\in G.$$
By using (\ref{eq11}) we obtain
$$2f(z)\Psi(x,y)+2f(y)\Psi(x,z)=2f(z)[\int_{G}g(xty)d\mu(t)+\int_{G}g(\sigma(y)tx)d\mu(t)-2g(x)g(y)]$$
$$+2f(y)[\int_{G}g(xtz)d\mu(t)+\int_{G}g(\sigma(z)tx)d\mu(t)-2g(x)g(z)]$$
$$=\int_{G}\int_{G}f(zsxty)d\mu(t)d\mu(s)+\int_{G}\int_{G}f(zs\sigma(y)\sigma(t)\sigma(x))d\mu(t)d\mu(s)$$$$+
\int_{G}\int_{G}f(zs\sigma(y)tx)d\mu(t)d\mu(s)+\int_{G}\int_{G}f(zs\sigma(x)\sigma(t)y)d\mu(t)d\mu(s)$$
$$-\int_{G}\int_{G}f(ztxsy)d\mu(t)d\mu(s)-\int_{G}\int_{G}f(zsxt\sigma(y))d\mu(t)d\mu(s)$$
$$-\int_{G}\int_{G}f(zt\sigma(x)sy)d\mu(t)d\mu(s)-\int_{G}\int_{G}f(zt\sigma(x)s\sigma(y))d\mu(t)d\mu(s)$$
$$+\int_{G}\int_{G}f(ysxtz)d\mu(t)d\mu(s)+\int_{G}\int_{G}f(ys\sigma(z)\sigma(t)\sigma(x))d\mu(t)d\mu(s)$$
$$+\int_{G}\int_{G}f(ys\sigma(z)tx)d\mu(t)d\mu(s)+\int_{G}\int_{G}f(ys\sigma(x)\sigma(t)z)d\mu(t)d\mu(s)$$
$$-\int_{G}\int_{G}f(ytxsz)d\mu(t)d\mu(s)-\int_{G}\int_{G}f(ytxs\sigma(z))d\mu(t)d\mu(s)$$
$$-\int_{G}\int_{G}f(yt\sigma(x)sz)d\mu(t)d\mu(s)-\int_{G}\int_{G}f(yt\sigma(x)s\sigma(z))d\mu(t)d\mu(s)$$
$$=\int_{G}\int_{G}f(zt\sigma(y)s\sigma(x))d\mu(t)d\mu(s)+\int_{G}\int_{G}f(zt\sigma(y)sx)d\mu(t)d\mu(s)$$
$$+\int_{G}\int_{G}f(yt\sigma(z)s\sigma(x))d\mu(t)d\mu(s)+\int_{G}\int_{G}f(yt\sigma(z)sx)d\mu(t)d\mu(s)$$
$$-\int_{G}\int_{G}f(zsxt\sigma(y))d\mu(t)d\mu(s)-\int_{G}\int_{G}f(zs\sigma(x)t\sigma(y))d\mu(t)d\mu(s)$$
$$-\int_{G}\int_{G}f(ytxs\sigma(z))d\mu(t)d\mu(s)-\int_{G}\int_{G}f(yt\sigma(x)s\sigma(z))d\mu(t)d\mu(s),$$
where the last identity is due to our assumption that $\mu$ is
$\sigma$-invariant. Using  equation (\ref{eq11}) and the above
computation to obtain $$2f(z)\Psi(x,y)+2f(y)\Psi(x,z)$$
$$=2g(x)\int_{G}f(zt\sigma(y))d\mu(t)+2g(x)\int_{G}f(yt\sigma(z))d\mu(t)$$
$$-\int_{G}\int_{G}f(ytxs\sigma(z))d\mu(t)d\mu(s)-\int_{G}\int_{G}f(zs\sigma(x)t\sigma(y))d\mu(t)d\mu(s)$$
$$-\int_{G}\int_{G}f(zsxt\sigma(y))d\mu(t)d\mu(s)-\int_{G}\int_{G}f(yt\sigma(x)s\sigma(z))d\mu(t)d\mu(s)$$
$$=2g(x)\int_{G}f(zt\sigma(y))d\mu(t)-2g(x)\int_{G}f(zt\sigma(y))d\mu(t)$$
$$+\int_{G}\int_{G}f(zs\sigma(x)t\sigma(y))d\mu(t)d\mu(s)-\int_{G}\int_{G}f(zs\sigma(x)t\sigma(y))d\mu(t)d\mu(s)$$
$$+\int_{G}\int_{G}f(yt\sigma(x)s\sigma(z))d\mu(t)d\mu(s)-\int_{G}\int_{G}f(yt\sigma(x)s\sigma(z))d\mu(t)d\mu(s)=0,$$
 which is due to assumptions that $f$ is odd and $\mu$ is $\sigma$-invariant. So, this implies that
$f(z)\Psi(x,y)+f(y)\Psi(x,z)=0$ for all $x,y,z\in G$ and then we
conclude that there exists $c_x \in \mathbb{C}$ such that $\Psi(x,y)
= c_xf(y)$, $c_xf(z)f(y)+c_xf(y)f(z) = 0$. We get for any $x,y \in
G$ $\Psi(x,y)=0$. Now, since $g$ is even: $g(\sigma(x))=g(x)$ for
all $x\in G$ and $\mu$ is $\sigma$-invariant we obtain
$$2g(x)g(y)=\int_{G}g(xty)d\mu(t)+\int_{G}g(\sigma(y)tx)d\mu(t)$$
$$=\int_{G}g(\sigma(y)tx)d\mu(t)+\int_{G}g(\sigma(y)t\sigma(x))d\mu(t)=2g(\sigma(y))g(x).$$
This means that
$$\int_{G}g(xty)d\mu(t)+\int_{G}g(xt\sigma(y))d\mu(t)=2g(x)g(y)$$
for all $x,y\in G$ and this completes the proof.
\begin{thm}Let $\sigma$ be a continuous involution of $G$. Let $\mu$
be a complex measure with compact support such $\mu\ast\mu=\mu$ and
$\mu$ is $\sigma$-invariant. If the pair $f,g$: $G\longrightarrow
\mathbb{C}$, where  $f\neq 0$ is a  continuous solution of the
generalized Wilson's functional equation (\ref{eq11}) such that
$f_\mu\neq 0.$  Then $g$ is a solution of the generalized
d'Alembert's short functional equation (\ref{eq12}).
\end{thm} In the proof we use the ideas of Stetk\ae r   \cite{st3}. Let $g$ be a non zero fixed solution of
equation (\ref{eq11}). \\We put $W_g$: $=\{f: G\longrightarrow
\mathbb{C}| \;f\; \text{is continuous, satisfies } \;(\ref{eq11}),\;
f_\mu=f\;\text{and}\; f(e)=0\}$. A continuous solution $f$ of
equation (\ref{eq11}) such that $f_\mu=f$  is odd iff $f(e)=0$.
\\If $W_g\neq \{0\}$ then we get the result from Theorem 2.1. Assume
now that $W_g=\{0\}$. Let $f$ be a non zero solution of equation
(\ref{eq11}) such that $f_\mu\neq 0$. The function $f_\mu$ is also a
nonzero solution of equation (\ref{eq11}). Since $W_g=\{0\}$,
$(f_\mu)_\mu=f_\mu$, then $f_\mu(e)\neq0$. Replacing $f_\mu$ by
$f_\mu/f_\mu(e)$ we may assume that $f_\mu(e)=1$. If $h$ is a
continuous solution of equation (\ref{eq11}), then
$h_\mu-h_\mu(e)f_\mu\in W_g=\{0\}$, so $h_\mu=h_\mu(e)f_\mu$. Let
$x\in G$, since $\delta(x)f_\mu(y)=\int_{G}f_\mu(xty)d\mu(t)$ is a
solution of equation (\ref{eq11}), $\mu\ast\mu=\mu$ and
$(\delta(x)f_\mu)_\mu=\delta(x)f_\mu$, then there exists $\psi(x)$
such that $\delta(x)f_\mu=\psi(x)f_\mu$.  In particular for $y=e$ we
have
$(\psi(x)f_\mu)(e)=\psi(x)=(\delta(x)f_\mu)(e)=\int_{G}f_{\mu}(xt)d\mu(t)$,
so we get
\begin{equation}\label{eq900}
\int_{G}f_\mu(xty)d\mu(t)=f_\mu(y)\int_{G}f_{\mu}(xt)d\mu(t)\end{equation}
for all $x,y\in G.$\\ Now, we will show that
$\int_{G}f_{\mu}(xt)d\mu(t)=f_\mu(x)$ for all $x\in G$. Since
$f_\mu$ satisfies (\ref{eq11}) and $\mu\ast\mu=\mu$ then we get
$$\int_{G}f_\mu(xty)d\mu(t)+\int_{G}f_\mu(xt\sigma(y))d\mu(t)=2f_\mu(x)g(y)$$
$$=\int_{G}\int_{G}f_\mu(xsty)d\mu(s)d\mu(t)+\int_{G}\int_{G}f_\mu(xst\sigma(y))d\mu(s)d\mu(t)=2\int_{G}f_\mu(xs)d\mu(s)g(y),$$
which implies that $\int_{G}f_\mu(xs)d\mu(s)=f_\mu(x)$, for all
$x\in G$ and then equation (\ref{eq900}) can be written as follows.
$\int_{G}f_\mu(xty)d\mu(t)=f_{\mu}(x)f_\mu(y)$ for all $x,y\in G.$
Substituting this result into
$$\int_{G}f_\mu(xty)d\mu(t)+\int_{G}f_\mu(xt\sigma(y))d\mu(t)=2f_\mu(x)g(y)$$ we
get $g(y)=\frac{f_\mu(y)+f_\mu(\sigma(y))}{2}$ and we can easily
verify that $g$ satisfies the generalized d'Alembert's short
functional equation (\ref{eq12}). This ends the proof of
theorem.\end{proof} \section{Hyers-Ulam stability of Wilson's
functional equation}In [\cite{elq2}, Corollary 2.7 (iii)] the
authors proved that if the function
$$(x,y)\longrightarrow
\int_{G}f(xty)d\mu(t)+\int_{G}f(xt\sigma(y))d\mu(t)-2f(x)g(y)$$ is
bounded and $f$ is unbounded then $g$ is a solution of the
generalized  d'Alemebert's long functional equation (\ref{eq13}). In
the following theorem under another kind of assumption we get that
$g$ is a solution of the generalized  d'Alembert's short functional
equation (\ref{eq12}).
\begin{thm}Let $\sigma$ be a continuous involution of $G$.
Let $\mu$ be a discrete complex measure with compact support such
that $\mu$ is $\sigma$-invariant and $\mu\ast\mu=\mu$. Let
$\delta\geq 0$. If the pair $f,g$: $G\longrightarrow \mathbb{C}$,
where f is an unbounded $\mu$-biinvariant  continuous solution of
the following inequality
\begin{equation}\label{eq400}
    \mid\int_{G}f(xty)d\mu(t)+\int_{G}f(xt\sigma(y))d\mu(t)-2f(x)g(y)\mid\leq\delta
\end{equation} for all $x,y\in G$. Then $g$ is a solution of the
generalized  d'Alembert's short functional equation (\ref{eq12})
\end{thm}\begin{proof} Assume that $f,g$ satisfy inequality (\ref{eq400}) where  $f$ is unbounded on $G$. So, for all $x,y\in G$ we have
$$\mid\int_{G}f(xty)d\mu(t)+\int_{G}f(xt\sigma(y))d\mu(t)-2f(x)g(y)\mid\leq\delta,$$
$$\mid\int_{G}f(xt\sigma(y))d\mu(t)+\int_{G}f(xty)d\mu(t)-2f(x)g(\sigma(y))\mid\leq\delta,$$
and by triangle inequality we find
$$|2f(x)||g(y)-g(\sigma(y))|\leq2\delta$$ for all $x,y\in G$. Since
$f$ is assumed to be unbounded, then we get $g(\sigma(y))=g(y)$ for
all $y\in G.$ In the rest of the proof we use some ideas of the
proof of Theorem 2.1 and Theorem 2.2.\\ First Case: We assume that
the function $x\longmapsto f(x)+f(\sigma(x))$ is a bounded function
on $G$, that is $|f(x)+f(\sigma(x))|\leq\beta$ for some $\beta\geq0$
and for all $x\in G.$ Let
$\Psi(x,y)=\int_{G}g(xty)d\mu(t)+\int_{G}g(\sigma(y)tx)d\mu(t)-2g(x)g(y),\;
x,y\in G.$ We will show  that the function $(x,y,z)\longmapsto
2f(z)\Psi(x,y)+2f(y)\Psi(x,z)$ is a bounded function when $g$ is
bounded. The above computations show that
$$2f(z)\Psi(x,y)+2f(y)\Psi(x,z)$$
$$=2g(x)[\int_{G}f(zt\sigma(y))d\mu(t)+\int_{G}f(yt\sigma(z))d\mu(t)]$$
$$-\int_{G}\int_{G}f(ytxs\sigma(z))d\mu(t)d\mu(s)-\int_{G}\int_{G}f(zs\sigma(x)t\sigma(y))d\mu(t)d\mu(s)$$
$$-\int_{G}\int_{G}f(zsxt\sigma(y))d\mu(t)d\mu(s)-\int_{G}\int_{G}f(yt\sigma(x)s\sigma(z))d\mu(t)d\mu(s).$$
So, we get
\begin{equation}\label{eq300}
|2f(z)\Psi(x,y)+2f(y)\Psi(x,z)|\leq
2\beta\|\mu\||g(x)|+2\beta\|\mu\|^{2}\end{equation}
 for all $x,y,z\in
G$, where $\|\mu\|=Sup\{|<f,\mu>|,\; f\in C(G)\;\|f\|_{\infty}=1\}$.
There are two possibilities. One is:  $g$ is unbounded, then from
[\cite{elq2}, Corollary 2.7 (iii)] and Theorem 2.2, we conclude that
$g$ is a solution of the generalized d'Alembert's short functional
equation (\ref{eq11}).
\\The other possibility is: $g$ is a bounded function, then from (\ref{eq300}) the
function $(x,y,z)\longmapsto 2f(z)\Psi(x,y)+2f(y)\Psi(x,z)$ is a
bounded function on $G$. Having assumed $f$ unbounded, this implies
that then there exists a sequence $(z_n)_{n\in \mathbb{N}}$ such
that $\lim_{n\longrightarrow+\infty}|f(z_n)|=+\infty$. By using
(\ref{eq300}), there exists $c_x\in \mathbb{C}$ such that
$\Psi(x,y)=c_xf(y)$, and  the function $(x,y,z)\longrightarrow
2f(z)c_xf(y)+2f(y)c_xf(z)$ is bounded. By using again the
unboundedness of $f$ we get $c_x=0$ for all $x\in G$. This means
that $2f(z)\Psi(x,y)+2f(y)\Psi(x,z)=0$ for all $x,y,z\in G.$ The
computations above shows that $g$ is a solution of the generalized
 d'Alembert's short functional equation (\ref{eq12}). \\Now, let  $f,g$  be
 functions such that $f$ is unbounded on $G$ and the function $(x,y)\longrightarrow
\int_{G}f(xty)d\mu(t)+\int_{G}f(xt\sigma(y))d\mu(t)-2f(x)g(y)$ is
bounded on $G\times G$. One can verify that the function:
$(x,y)\longrightarrow
\int_{G}f_{\mu}(xty)d\mu(t)+\int_{G}f_{\mu}(xt\sigma(y))d\mu(t)-2f_{\mu}(x)g(y)$
is bounded on $G\times G$. If $f_{\mu}$ is unbounded and
$f_{\mu}(e)=0$, then since $\mu\ast\mu=\mu$ we have
$f_{\mu}(x)+f_{\mu}(\sigma(x))$ is a bounded function, so by using
the precedent proof, we get $g$ satisfies equation (\ref{eq12}). For
the rest of the proof  we fixe $g$ and we assume that
$f_{\mu}(e)\neq0$.
 Replacing $f_{\mu}$ by $f_{\mu}/f_{\mu}(e)$ we may assume that
$f_{\mu}(e)=1$. Consider the  function
$\delta_af(x)=\int_{G}f(atx)d\mu(t)$, $x\in G$. We can easily verify
that the function $(x,y)\longrightarrow
\int_{G}\delta_af(xty)d\mu(t)+\int_{G}\delta_af(xt\sigma(y))d\mu(t)-2\delta_af(x)g(y)$
is bounded on $G\times G$.\\ If there exists $a\in G$ such that
$h=(\delta_af)_\mu-(\delta_af)_{\mu}(e)f_{\mu}$ is unbounded on $G$,
since $h(e)=0$, $h_\mu=h$, the function $(x,y)\longrightarrow
\int_{G}h(xty)d\mu(t)+\int_{G}h(xt\sigma(y))d\mu(t)-2h(x)g(y)$ is
bounded we get $x\longrightarrow h(x)+h(\sigma(x))$ is a bounded
function on $G$. Then from above computations we get that $g$ is a
solution of equation (\ref{eq12}). Now, assume that for all $x\in G$
the function $y\longrightarrow
(\delta_xf)_{\mu}(y)-(\delta_xf)_{\mu}(e)f_{\mu}(y)$ is bounded,
that is there exists $M(x)$ such that
\begin{equation}\label{eq51}
|\int_{G}f(xty)d\mu(t)-\int_{G}f(xt)d\mu(t)\int_{G}f(ty)d\mu(t)|\leq
M(x)
\end{equation} for all $x,y\in G.$ Since $f$ is assumed to be $\mu$-biinvariant, that is  $\int_{G}f(xt)d\mu(t)=\int_{G}f(tx)d\mu(t)=f(x)$ for all $x\in G$
 then inequality (\ref{eq51}) can be replaced by
\begin{equation}\label{eq52}
|\int_{G}f(xty)d\mu(t)-f(x)f(y)|\leq M(x)
\end{equation} for all $x,y\in G.$ Indeed in view of  triangle inequality  we
get for all $x,y,z \in G$ that
$$|f(z)||\int_{G}f(xty)d\mu(t)-f(x)f(y)|\leq|-\int_{G}\int_{G}f(xtysz)d\mu(t)d\mu(s)+\int_{G}f(xty)d\mu(t)f(z)|$$
$$+|\int_{G}\int_{G}f(xtysz)d\mu(t)d\mu(s)-f(x)\int_{G}f(ysz)d\mu(s)|+|f(x)|\int_{G}f(ysz)d\mu(s)-f(y)f(z)|$$
$$\leq\int_{G}M(xty)d|\mu|(t)+\|\mu\|M(x)+|f(x)|M(y).$$Since $f$ is
assumed to be unbounded, then we get
$\int_{G}f(xty)d\mu(t)=f(x)f(y)$ for all $x,y\in G.$ Substituting
this result into inequality (\ref{eq400}) we get that
$$|f(x)||f(y)+f(\sigma(y))-2g(y)|\leq \delta$$ for all $x,y\in G$.
Since $f$ is unbounded then we obtain
$g(y)=\frac{f(y)+f(\sigma(y))}{2}$ for all $y\in G$, from which a
simple computation shows that $g$ is a solution of the generalized
d'Alembert's short functional equation (\ref{eq12}). This completes
the proof.
\end{proof}The first part of the above proof proves the following
corollary.
\begin{cor}
Let $\sigma$ be a continuous involution of $G$. Let $\mu$ be a
complex measure with compact support such that $\mu$ is
$\sigma$-invariant and $\mu\ast\mu=\mu$. Let $\delta\geq 0$. If the
pair $f,g$: $G\longrightarrow \mathbb{C}$, where $f$ is an unbounded
continuous solution of the following inequality
\begin{equation}\label{eq31}
    \mid\int_{G}f(xty)d\mu(t)+\int_{G}f(xt\sigma(y))d\mu(t)-2f(x)g(y)\mid\leq\delta
\end{equation} for all $x,y\in G$ and such that $x\longrightarrow f(x)+f(\sigma(x))$ is a bounded function. Then $g$ is a solution of the
generalized  d'Alembert's short functional equation (\ref{eq12}).
\end{cor}
\begin{cor}
Let $\sigma$ be a continuous involution of $G$. Let $\mu$ be a
complex measure with compact support such that $\mu$ is
$\sigma$-invariant and $\mu\ast\mu=\mu$. Let $\delta\geq 0$. If the
pair $f,g$: $G\longrightarrow \mathbb{C}$, where $g$ is an unbounded
continuous solution of the following inequality
\begin{equation}\label{eq31}
    \mid\int_{G}f(xty)d\mu(t)+\int_{G}f(xt\sigma(y))d\mu(t)-2f(x)g(y)\mid\leq\delta
\end{equation} for all $x,y\in G$. Then the pair $f,g$ is  a continuous  solution of the functional equation
(\ref{eq11}). Furthermore, if $f\neq 0$ and $f_\mu\neq 0$, then $g$
satisfies the geneleralized d'Alembert's short functional equation
(\ref{eq12}).
\end{cor}\begin{proof}We use [\cite{elq2}, Corollary 2.7, (iii)] and Theorem 2.2.\end{proof}
In \cite{elq3} on general groups and under the hypotheses that the
function $(x,y)\longrightarrow f(xy)+f(x\sigma(y))-2f(x)g(y)$ is a
bounded function on $G\times G$, the function $g$ satisfies
d'Alembert's long functional equation
\begin{equation}\label{eq61}
g(xy)+g(x\sigma(y))+g(yx)+g(\sigma(y)x)=4g(x)g(y),\;x,y\in
G\end{equation}
 The following corollaries finishes the work on the
Hyers Ulam stability of Wilson's functional equation
\begin{equation}\label{eq62}
f(xy)+f(x\sigma(y))=2f(x)g(y),\;x,y\in G \end{equation}
 on groups when
the solutions are unbounded functions.\\
If we let $\mu=\delta_e$: The dirac measure concentrated at the
identity element of G, we apply Theorem 2.3 to $\mu=\delta_e$ to
obtain the following result which has been proved by several authors
in the case where $G$ is an abelian group.
\begin{cor} Let $G$ be a group.
Let $\sigma$ be an involution of $G$.  Let $\delta\geq 0$. If the
pair $f,g$: $G\longrightarrow \mathbb{C}$, where $f$ is an unbounded
solution of the following inequality
\begin{equation}\label{eq54}
    \mid f(xy)+f(x\sigma(y))-2f(x)g(y)\mid\leq\delta
\end{equation} for all $x,y\in G$. Then $g$ is a solution of  d'Alembert's short functional equation
\begin{equation}\label{eq63}
g(xy)+g(x\sigma(y))=2g(x)g(y),\;x,y\in G \end{equation}
\end{cor}Now, we can ready to formulate the Hyers-Ulam stability of
the classical Wilson's functional equation (\ref{eq62}) on groups.
The following result was obtained by Kannappan and Kim \cite{k2}
under the condition that $f$ is even and $f$ satisfies the Kannappan
condition $f(xyz)=f(yxz)$ for all $x,y,z\in G$.
\begin{cor}
Let  $\delta\geq 0$, $G$  a group,  $\sigma$ an involution of $G$.
Suppose that the pair $f,g:G\rightarrow \mathbb{C}$ satisfies
\begin{equation}\label{eq31}
|f(xy)+f(x\sigma(y))-2f(x)g(y)|\leq \delta,\;\text{for all}\;x,y\in
G.
\end{equation} Under these assumptions the following statements
hold:\\
(1) If $f$ is unbounded, then $g$ satisfies d'Alembert's short
functional equation (\ref{eq63})\\ (2) If $g$ is unbounded and
$f\neq 0$ then the pair ($f,g$) satisfies  Wilson's  functional
equation (\ref{eq62}) and  $g$ satisfies  d'Alembert's short
functional equation (\ref{eq63}).
\end{cor}\begin{proof} We use [\cite{elq3}, Theorem 2.2] and Theorem 2.2.\end{proof}
Recently, the authors proved the following result.
\begin{thm}\cite{elq3} \label{thm21} Let  $\delta\geq 0$, $G$  a group, $\chi$  a unitary
character of $G$ and $\sigma$  an involution of $G$ such that
$\chi(x\sigma(x))=1$ for all $x\in G$. Suppose that the pair
$f,g:G\rightarrow \mathbb{C}$ satisfies
\begin{equation}\label{eq700}
|f(xy)+\chi(y)f(x\sigma(y))-2f(x)g(y)|\leq \delta,\;\text{for all
}\;x,y\in G.
\end{equation} Under these assumptions the following statments
hold:\\
 (a) If $f$ is unbounded then $g$ satisfies d'Alembert's long functional
equation
\begin{equation}\label{eq701}
g(xy)+\chi(y)g(x\sigma(y))+g(yx)+\chi(y)g(\sigma(y)x)=4g(x)g(y),\;x,y\in
G\end{equation}  (b) If  $g$ is unbounded and $f\neq0$, then the
pair ($f,g$) satisfies Wilson's functional equation
\begin{equation}\label{eq702}
f(xy)+\chi(y)f(x\sigma(y))=2f(x)g(y)\; x,y\in G \end{equation} and
$g$ satisfies the d'Alembert's short functional equation
\begin{equation}\label{eq703} g(xy)+\chi(y)g(x\sigma(y))=2g(x)g(y)\;
x,y\in G. \end{equation}
\end{thm} The purpose in  the following is to prove  that in Theorem 3.6,  case (a),  the function $g$
satisfies  d'Alembert's short functional equation (\ref{eq703}). We
notice here that the solutions of equation (\ref{eq703}) are
obtained by Stetk\ae r in \cite{st7}.
\begin{prop} Let $G$ be a group. Let $\delta\geq 0$.
 Let $\chi$  a unitary
character of $G$ and $\sigma$  an involution of $G$ such that
$\chi(x\sigma(x))=1$ for all $x\in G$.  If the pair $f,g$:
$G\longrightarrow \mathbb{C}$, where $f$ is an unbounded solution of
the following inequality
\begin{equation}\label{eq54}
    \mid f(xy)+\chi(y)f(x\sigma(y))-2f(x)g(y)\mid\leq\delta
\end{equation} for all $x,y\in G$. Then $g$ is a solution of  d'Alembert's short functional equation
(\ref{eq703}). \end{prop}
\begin{proof} In the proof we use similar reasoning to that in the proof of Theorem 2.3 and some  author's
computations \cite{elq3}. Assume that the pair $f,g$ satisfies
inequality (\ref{eq54}) where $f$ is an unbounded function on $G$.
First Case: We assume that the function $x\longmapsto
f(x)+\chi(x)f(\sigma(x))$ is a bounded function on $G$, that is
$|f(x)+\chi(x)f(\sigma(x))|\leq\beta$ for some $\beta\geq0$ and for
all $x\in G.$ Let $\Psi(x,y)=g(xy)+\chi(y)g(\sigma(y)x)-2g(x)g(y),\;
x,y\in G.$ We will show  that the function $(x,y,z)\longmapsto
2f(z)\Psi(x,y)+2f(y)\Psi(x,z)$ is a bounded function in the special
case when $g$ is bounded. The  computations in \cite{elq3} show that
$$2f(z)\Psi(x,y)+2f(y)\Psi(x,z)$$
$$=\mu(y)2f(z\sigma(y))g(x)+\mu(z)2f(y\sigma(z))g(x)-\mu(z)f(yx\sigma(z))$$
$$-\mu(x)\mu(y)f(z\sigma(x)\sigma(y))-\mu(y)f(zx\sigma(y))-\mu(x)\mu(z)f(y\sigma(x)\sigma(z))$$
So, we get
\begin{equation}\label{eq41}
|2f(z)\Psi(x,y)+2f(y)\Psi(x,z)|\leq
2\beta\|\mu\||g(x)|+2\beta\|\mu\|^{2}\end{equation}
 for all $x,y,z\in
G$. Now, we will discuss two subcases:\\ If $g$ is unbounded then
from  [\cite{elq3}, Theorem 2.2, (b)], we conclude that $g$ is a
solution of the generalized d'Alembert's short functional equation
(\ref{eq703}).
\\In the rest of the proof we examine the case of $g$  a bounded function on $G$. Then from (\ref{eq41}) the
function $(x,y,z)\longmapsto 2f(z)\Psi(x,y)+2f(y)\Psi(x,z)$ is a
bounded function on $G$. By using similar computations in the proof
of Theorem 2.3 we get our result and this completes the proof.
\end{proof}

Authors'addresses:\\
 Bouikhalene Belaid,   Polydisciplinary Faculty,
University Sultan Moulay Slimane, Beni-Mellal, Morocco, E-mail :
bbouikhalene@yahoo.fr\\
\\Elqorachi Elhoucien, Department of Mathematics,
Faculty of Sciences, University  Ibn Zohr, Agadir, Morocco, E-mail:
elqorachi@hotamail.com
\end{document}